\documentclass[11pt]{amsart}

\usepackage{graphicx}
\usepackage{amsmath}
\usepackage{amsfonts}
\usepackage{amssymb}
\usepackage{xypic}
\usepackage[all]{xy}
\usepackage{array}
\usepackage{enumerate}
\usepackage{url}

\newtheorem{theorem}{Theorem}[section]

\newtheorem{example}[theorem]{Example}

\newtheorem{lemma}[theorem]{Lemma}

\newtheorem{proposition}[theorem]{Proposition}

\newcommand{\DD}{\mathcal{D}}

\newcommand{\darrow}{{\downarrow}}

\newcommand{\restricted}[2]{#1{\mid}_{#2}}

\def\dom{\mathrm{dom}}

\title{Noncommutative frames}
\author{Karin Cvetko-Vah}
\address{ University of Ljubljana \\Faculty of Mathematics and Physics\\Jadranska 19\\1000 Ljubljana \\Slovenia}

\begin{document}
\maketitle

\begin{abstract}
We explore algebraic properties of noncommutative frames. The concept of noncommutative frame is due to  Le Bruyn, who introduced it in connection with noncommutative covers of the Connes-Consani arithmetic site.
\end{abstract}

\section{Introduction}

The motivation for our definition of a noncommutative frame comes from the following interesting example of noncommutative covers on the arithmetic site that is due to Le Bruyn \cite{lebruyn2016}. We refer the reader to \cite{connes2014} and \cite{connes2016} for the definitions of an arithmetic site, and to \cite{lebruyn2015} for the definition of the sieve topology on the arithmetic site.

\begin{example}[\cite{lebruyn2016}]\label{ex:lebruyn}
\emph{
The sieve topology on the arithmetic site is defined by the basic open sets that correspond to sieves $S\in \mathbf \Omega$ and are denoted by $\mathbb X_s(S)$, where $\mathbf \Omega$ is the subobject-classifier of the arithmetic site topos $\hat C$. The sheaf $\mathcal O^c$ of constructible truth fluctuations has as sections over the open set $\mathbb X_s(S)$ for $S\in \mathbf \omega$ the set of all continuous functions $x$ from $\mathbb X_s(S)$ (with the induced patch topology) to the Boolean semifield $\mathbb B=\{0,1\}$ (with the discrete topology).
A noncommutative frame $\mathbf \Theta$, which represents the set of opens of a noncommutative topology, is defined in \cite{lebruyn2016} as the set of all pairs $(S,x)$, where $S\in \mathbf \Omega$ and $x:\mathbb X_s(S)\to \mathbb B$ continuous. The noncommutative frame operations are defined on $\mathbf \Theta$ by:
\begin{eqnarray*}
(S,x)\land (T,y)=(S\land T, x|_{\mathbb X_s(S)\cap \mathbb X_s(T)}) \\
(S,x)\lor (T,y)=(S\lor T, y|_{\mathbb X_s(T)} \cup x|_{\mathbb X_s(S)- \mathbb X_s(T)} )\\
(S,x)\rightarrow (T,y)=(S\rightarrow T, y\cup 1_{\mathbb X_s(S\rightarrow T)- X_s(T)} ).\\
\end{eqnarray*}
}
\end{example}

Inspired by \cite{lebruyn2016}, in the present paper we define notions of noncommutative Heyting algebras and noncommutative frames. A careful reader will notice that our definition of a noncommutative frame differs from Le Bruyn's definition given in the above example in that we  require our algebras to be join complete and satisfy infinite distributive laws.
Although the present paper provides algebraic investigation into the structures used in \cite{lebruyn2016}, it makes no use of algebraic-geometric and topos-theoretic concepts. We refer the reader who is interested in applications of noncommutative frames to noncommutative generalizations of topoi to \cite{CV2017}.

 Our work belongs to the context of skew lattices which are noncommutative generalizations of lattices. Skew lattices were introduced in 1949 by the mathematical physicist Pascual Jordan, a co-founder of the quantum field theory, who was exploring mathematical enviroments suitable to encode the logic of quantum mechanics \cite{jordan}. The theory of skew lattices was revived in late 80's and early 90's when Jonathan Leech wrote a series of papers on the subject, beginning with \cite{L1}. Leech's motivation came mainly from the algebraic ring  theory.  Leech's survey paper \cite{L3} is an excellent source for learning the basics of  skew lattice theory.

The paper is structured as follows. We begin Section 2  with the definition of a skew lattice, followed by  other important definitions and results. 
In Section 3 we define noncommutative Heyting algebras and explore their algebraic properties. In Section 4 we define noncommutative frames that generalize the usual notion of a frame. Theorem 4.4 establishes the relation between frames and noncommutative frames. Theorem 4.5 yields that noncommutative frames are  precisely join complete noncommutative Heyting algebras that satisfy the infinite distributive laws.

\section{Preliminaries}

Following \cite{L1}, a \emph{skew lattice} is an algebra $(S;\land, \lor)$ where
$\wedge$ and $\vee$ are   idempotent and associative binary operations that satisfy the absorption laws
$x\land (x\lor y)=x=x\lor (x\land y)$ { and } $(x\land y)\lor y=y=(x\lor y)\land y$.  Given a skew lattice $S$ and $x,y\in S$ the following equivalences hold: 
\begin{eqnarray*}
x\land y=x \Leftrightarrow x\lor y=y \qquad \text{and} \qquad
 x\land y=y\Leftrightarrow x\lor y=x.
\end{eqnarray*}
A skew lattice is a lattice when both operations $\land$, $\lor$ are commutative. By a result of \cite{L1}, the operation $\land$ is commutative if and only if $\lor$ is such.

The \emph{natural partial order} is defined on a skew lattice $S$ by $x\leq y$ iff $x\land y=x=y\land x$, or equivalently $x\lor y =y=y\lor x$; and the \emph{natural preorder} is defined by $x\preceq y$ iff $x\land y\land x=x$, or equivalently $y\lor x\lor y =y$. The natural preorder induces Green's equivalence relation $\DD$ defined by $x\DD y$ iff $x\preceq y$ and $y\preceq x$. Leech's first decomposition theorem for skew lattices \cite{L1} yields that $\DD$ is a congruence on a skew lattice $S$, each congruence class is a rectangular subalgebra (characterized by $x\land y=y\lor x$) and $S/\DD$ is a maximal lattice image of $S$. We denote the $\DD$-class containing $x$ by $\DD_x$.

 A skew lattice is  \emph{strongly distributive} if it satisfies the identities:
\[(x\lor y)\land z=(x\land z)\lor (y\land z) \text{ and } x\land (y\lor z)=(x\land y)\lor (x\land z).
\]
Strongly distributive skew lattices are \emph{distributive} \cite{L1, L2}, ie. they satisfy the identities:
\begin{eqnarray}
x \land (y\lor z)\land x=(x\land y\land x)\lor (x\land z\land x), \\
\label{D2} x\lor (y\land z) \lor x=(x\lor y\lor x)\land (x\lor z\lor x).
\end{eqnarray}

A skew lattice $S$  is  \emph{symmetric} if given any $x,y\in S$, $x\land y=y\land x$ iff $x\lor y=y\lor x$.  A skew lattice is called \emph{normal} provided that it satisfies the identity $x\land y\land z\land x=x\land z\land y\land x$.  

\begin{lemma}[\cite{bl}] \label{lemma:sd-equiv-def}
A skew lattice is  strongly distributive if and only if it is symmetric and normal and $S/\DD$ is a distributive lattice. 
\end{lemma}

\begin{lemma}[\cite{L4}] \label{lemma:normal-lattice}
A skew lattice is normal if and only if $u\darrow =\{x\in S\;|\; x\leq u\}$ is a lattice for all $u\in S$. 
\end{lemma}

If $S$ is a symmetric skew lattice then we say that elements $x$ and $y$ in $S$ \emph{commute} if $x\land y=y\land x$ (and hence also $x\lor y=y\lor x$); a subset $A\subseteq S$ is a \emph{commuting subset} if $x$ and $y$ commute for all  $x,y\in A$.
 A \emph{lattice section} $L$ of a skew lattice $S$ is a sub-algebra that is a lattice (ie. both $\land$ and $\lor$ are commutative on $L$) which intersects each $\DD$-class in exactly one element.  A lattice section (when it exists) of a skew lattice is a maximal commuting subset   isomorphic to its maximal lattice image by a result of \cite{L1}. If a normal skew lattice $S$ has a top $\DD$-class $T$ then given $t\in T$,  $t\darrow =\{x\in S\;|\; x\leq t\}$ is a lattice section of $S$. 
 A skew lattice has a bottom element $0$ if $x\lor 0=x=x\lor 0$, or equivalently $x\land 0=0=0\land x$) holds for all $x\in S$.

A connection between strongly distributive skew lattices and sheaves over distributive skew lattices was established in \cite{skew-priestley}, where it was shown that the category of  left-handed (ie. satisfying $x\land y\land x=x\land y$, or equivalently $x\lor y\lor x=y\lor x$) strongly distributive skew lattice with $0$ is dual to sheaves over locally compact Priestley spaces with suitable morphisms. Via this duality the elements of the skew lattice $S$ are represented as  sections over compact and open subsets of the Priestley space of the distributive lattice $S/\DD$. Given two such sections $s$ and $r$ the skew lattice operations are defined by:
\[\begin{array}{ll}
  \text{\emph{Restriction: }} & s \land r = \restricted{s}{\dom{s}\cap \dom{r}}.\\
  \text{\emph{Override: }} & s \lor r = r\cup \restricted{s}{\dom{s}\setminus \dom{r}}.\\
\end{array}
\]
We refer the reader to \cite{P} and \cite{P1} for the definition and further details regarding the Priestley space of a distributive lattice.

A \emph{Heyting algebra} is an algebra $(H;\land,\lor, \rightarrow, 1, 0)$ such that $(H,\land,\lor, 1, 0)$ is a bounded distributive lattice that satisfies  the following set of axioms:
\begin{itemize}
     \item[(H1)] $(x\rightarrow x)=1$,
     \item[(H2)] $x\land (x\rightarrow y)=x\land y$,
     \item[(H3)] $y\land (x\rightarrow y)=y$,
     \item[(H4)] $x\rightarrow (y\land z)=(x\rightarrow y)\land (x\rightarrow z).$
\end{itemize}
Equivalently, the axioms (H1)--(H4) can be replaced by the following single axiom:
\begin{itemize}
   \item[(HA)] $x\land y\leq z$ iff $x\leq y\rightarrow z$.
 \end{itemize}

A \emph{frame} is a lattice that has all finite meets and all joins (finite and infinite), and satisfies the infinite distributive law:
\[x\land \bigvee_i y_i=\bigvee_i(x\land y_i).
\]
Frames are exactly complete Heyting algebras, see \cite{MM} for details.
 
\section{Noncommutative Heyting algebras}

\begin{lemma}\label{lemma:sd-top-comm}
If a strongly distributive skew lattice $S$ contains a top element $1$ (satisfying $1\land x=x=x\land 1$ and $1\lor x=1=x\lor 1$) then it is commutative.
\end{lemma}

\begin{proof}
$S$ is normal by Lemma \ref{lemma:sd-equiv-def}. Lemma \ref{lemma:normal-lattice} implies that $1\darrow$ is a lattice. Obviously, $S=1\darrow$.
\end{proof}

Hence one needs to sacrifice the top element when passing to the noncommutative setting. (Alternatively, one could deal with order-duals of strongly distributive skew lattices and keep $1$ but sacrifice $0$. The latter approach is reasonable when logic is considered, and was carried out in \cite{SHA}.)

A \emph{noncommutative Heyting algebra} is an algebra $(S;  \land, \lor, \rightarrow, 0, t)$ where $(S; \land, \lor, 0)$ is a strongly distributive skew lattice with bottom $0$ and a top $\DD$-class $T$, $t$ is a distinguished element of $T$ and $\rightarrow$ is a binary operation  that satisfies the following axioms:
\begin{itemize}
\item[(NH1)] $x\rightarrow y = (y\lor(t\land  x\land t)\lor y)\rightarrow y$,
\item[(NH2)] $x\rightarrow x = x\lor t\lor x$,
\item[(NH3)]  $x\land (x\rightarrow y) \land x = x\land y \land x$, 
  \item[(NH4)] $y\land (x\rightarrow y)=y$ and  $(x\rightarrow y)\land y=y$,
  \item[(NH5)] $x \rightarrow (t\land (y\land z)\land t)=
  ( x\rightarrow (t\land y\land t)) \land (x \rightarrow (t \land  z\land t))$.
\end{itemize}

Note that the axiom (NH4) yields $y\leq x\rightarrow y$ for all $x,y\in S$.
Noncommutative Heyting algebras form a variety because the fact that the distinguished element $t$ lies in the top $\DD$-class is characterized by $x\land t\land x=x$ (or equivalently, $t\lor x\lor t=t$).

\begin{example}
\emph{
Let $\mathcal P(A,\{0,1\})$ be the set of all partial functions from $A$ to $\{0,1\}$ where $A$ is a non-empty set.  Leech \cite{L4} defined skew lattice operations on $\mathcal P(A,\{0,1\})$ by:
\begin{eqnarray*}
f\land g=f|_{\dom f\cap \dom g}\\
f\lor g=g\cup f|_{\dom f-\dom g}\\
\end{eqnarray*}
Leech proved that  $(\mathcal P(A,\{0,1\};\land,\lor)$ is a left-handed, strongly distributive skew lattice with bottom $\emptyset$, with the maximal lattice image $\mathcal P(A,\{0,1\})/\DD$ being isomorphic to the power set of $A$. The top $\DD$-class of  $\mathcal P(A,\{0,1\})$ consists of all total functions. Denote by $\tau$ the total function defined by $\tau(x)=1$ for all $x\in A$.
 Following Le Bruyn's Example \ref{ex:lebruyn} we define the operation $\rightarrow$ on $\mathcal P(A,\{0,1\})$ by:
\[f\rightarrow g=g\cup \tau|_{A-(\dom f\cup \dom g)}.
\]
We claim that  $(\mathcal P(A,\{0,1\};\land,\lor,\rightarrow,\emptyset,\tau)$ is a noncommutative Heyting algebra.
}

\noindent \emph{ (NH1): $f\rightarrow g=(g\lor (\tau\land f\land \tau)\lor g)\rightarrow g$ because both sides of the equality reduce to $g\cup \tau|_{A-(\dom f\cup \dom g)}$.
}

\noindent \emph{ (NH2): $f\rightarrow f=f\lor \tau\lor f$ because both sides of the equality reduce to
 $f\cup \tau|_{A-\dom f}$.
}

\noindent \emph{ (NH3): $f\land (f\rightarrow g)\land f=f\land g\land f$ because both sides of the equality reduce to $f|_{\dom f\cap \dom g}$.
}

\noindent \emph{ (NH4): Both $g\land (f\rightarrow g)$ and $(f\rightarrow g)\land g$  reduce to $g$.
}

\noindent \emph{ (NH5): $f\rightarrow (\tau \land g\land h\land \tau)=(f\rightarrow (\tau \land g\land \tau))\land (f\rightarrow (\tau \land h\land \tau))$ because both sides of the equality reduce to $\tau|_{(A-\dom f)\cup (\dom g\cap \dom h)}$.
}
\end{example}

By a result of \cite{L1} skew lattices are \emph{regular}  in that they satisfy $x\land y\land x\land z\land x=x\land y\land z\land x$ and $x\lor y\lor x\lor z\lor x=x\lor y\lor z\lor x$.
The following is an easy but useful consequence of regularity.

\begin{lemma}[\cite{CC}]\label{lemma:reg}
Let $S$ be a skew lattice and $x,y,u,v\in S$ s.t. $u\preceq x,y\preceq v$ holds. Then:
\begin{itemize}
\item[(i)] $x\land v\land y= x\land y$,
\item[(ii)] $x\lor u\lor y=x\lor y$.
\end{itemize}
\end{lemma}

  We will make use of the following technical lemmas in the proof of Theorem \ref{th:phi-iso}.

\begin{lemma}\label{lemma:1.3}
Let  $(S;  \land, \lor, \rightarrow, 0, t)$ be a noncommutative Heyting algebra and let $x\in S$, $y,z\in t\darrow.$ Then:
\begin{itemize}
\item[(i)] $x\rightarrow t=t$,
\item[(ii)] $y\rightarrow z\in t\darrow$.
\end{itemize} 
Therefore the lattice $t\darrow$ is closed under $\rightarrow$.
\end{lemma}

\begin{proof}
(i) Using (NH1) we obtain: $x\rightarrow t=(t\lor (t\land x\land t)\lor t)\rightarrow t=t\rightarrow t=t$ by (NH2).

(ii) Since $z\leq t$ we obtain: $y\rightarrow z=y\rightarrow (t\land (z\land t)\land t)$. This is further equal to $y\rightarrow (t\land z\land t))\land  (y\rightarrow(t\land t\land t)$  by (NH5), which equals $(y\rightarrow z)\land  (y\rightarrow t)=(y\rightarrow z)\land t$. Similarly, we prove $t\land (y\rightarrow z)=y\rightarrow z$ and $y\rightarrow z\leq t$ follows.
\end{proof}

\begin{lemma}\label{lemma:d-cong}
Let  $(S;  \land, \lor, \rightarrow, 0, t)$ be a noncommutative Heyting algebra and $x,y\in S$. Then  $y$, $y\lor (t\land x\land t)\lor y$ and $x\rightarrow y$ all lie in the lattice $(y\lor t\lor y)\darrow$. 
\end{lemma}

\begin{proof}
Denote $t'=y\lor t\lor y$.
The absorption yields $y \lor (y\lor t\lor y)=y\lor t\lor y$ and likewise $(y\lor x\lor y)\lor y=y$. Hence $y\leq y\lor t\lor y$  and thus $y\in t'\darrow$.

Similarly, $(x\rightarrow y) \lor (y\lor t\lor y) =((x\rightarrow y)\lor y )\lor t\lor y=y\lor t\lor y$ since $y\leq (x\rightarrow y)$ by axiom (NH4). Together with $(y\lor t\lor y)\lor (x\rightarrow y)=y\lor t\lor y$ this yields  $x\rightarrow y\leq t'$. 

It remains to be proved that $y\lor (t\land x\land t)\lor y\leq t'$.  Using Lemma \ref{lemma:reg}: 
$(y\lor (t\land x\land t)\lor y)\lor t'=y\lor (t\land x\land t)\lor y\lor y\lor t\lor y=y\lor (t\land x\land t)\lor t\lor y$. By absorption this  equals $y\lor t\lor y=t'$. Likewise we prove that $t'\lor (y\lor (t\land x\land t)\lor y)=t'$ and $y\lor (t\land x\land t)\lor y\leq t'$ follows. 
\end{proof}

\begin{theorem}\label{th:phi-iso}
Let $(S;  \land, \lor, \rightarrow, 0, t)$ be a noncommutative Heyting algebra. Then:
\begin{itemize}
\item[(i)] $(t\darrow; \land, \lor, \rightarrow, 0, t)$ is  a Heyting algebra. 
\item[(ii)] Given any $t'\in \DD_t$,   $(t'\darrow; \land, \lor, \rightarrow, 0, t')$ is  a Heyting algebra, the map
\[\begin{array}{rcl}
\varphi: t\darrow & \to & t'\darrow \\
x &\mapsto & t'\land x\land t'
\end{array}
\]
is an isomorphism of Heyting algebras and $x\,\DD\,\varphi (x)$ holds for all $x\in t\darrow$.
\item[(iii)] Green's relation $\DD$ is a congruence on $S$ and the maximal lattice image $S/\DD$ is a Heyting algebra isomorphic to $t\darrow$.
\end{itemize}
\end{theorem}

\begin{proof}
(i) By a result of \cite{L4}  $t\darrow$ is a bounded distributive lattice, and by Lemma \ref{lemma:1.3} it is closed under $\rightarrow$. Since all elements of $t\darrow$ commute, given $x,y,u\in t\darrow$ the axioms (NH1) and (NH3)--(NH5) translate to the standard set of axioms of a Heyting algebra, while (NH2) translates to $(x\lor y)\rightarrow y=x\rightarrow y$, which follows from the axioms of a Heyting algebra.

(ii) Both $t\darrow$ and $t'\darrow$ are lattice sections of $S$. The map $\varphi$ is an isomorphism of  lattices by a result of \cite{L1} with an inverse given by $\psi(y)=t\land y\land t$. Hence $\varphi$ is also an isomorphism of Heyting algebras. 
% The implication on $t\darrow$ is characterized by the propery:
%\[x\leq y\rightarrow z \text{ iff } x\land y\leq z.
%\]
% Take any $x,y,z\in t\darrow$. Then:
%\[\begin{array}{rcl}
%x\leq y\rightarrow z & \text{iff} & x\land y\leq z \\
%&  \text{iff} & \varphi(x)\land \varphi(y)\leq \varphi(z) \\
%&  \text{iff} &  \varphi(x)\leq \varphi(y)\rightarrow \varphi(z).
%\end{array}
%\]
%Hence $\varphi(y\rightarrow z)=\varphi(y)\rightarrow \varphi(z)$ and $\varphi$ is a Heyting algebra isomorphism.
 Finally, $x\land \varphi (x)\land x=x\land t\land x\land t\land x=x$ by Lemma \ref{lemma:reg}, and likewise  $\varphi(x)\land x\land \varphi(x)=t'\land x\land t'\land x\land t'\land x\land t'=t'\land x\land t'=\varphi (x)$, which yields $x\,\DD\, \varphi (x)$.

(iii) By Leech's first decomposition theorem, $\DD$ is a congruence on any skew lattice. In order to prove that $\DD$ is a noncommutative Heyting algebra congruence, it remains to prove that it is compatible with $\rightarrow$. Assume $x\,\DD\, u$ and $y\,\DD \,v$. We need to prove that $x\rightarrow y\,\DD \,u\rightarrow v$. By Lemma \ref{lemma:d-cong}, $x\rightarrow y$ is an element of the Heyting algebra $(y\lor t\lor y)\darrow$ and  $u\rightarrow v$ is an element of the Heyting algebra $(v\lor t\lor v)\darrow$. We may assume that $y\leq x$, $x\in (y\lor t\lor y)\darrow$, $v\leq u$ and $u\in (v\lor t\lor v)\darrow$ (otherwise we replace $x$ by $y\lor (t\land x\land t)\lor y$ and $u$ by $v\lor (t\land u\land t)\lor v$). The map 
\[\begin{array}{rcl}
\rho: (y\lor t\lor y)\darrow & \to & (v\lor t\lor v)\darrow \\
z &\mapsto &  (v\lor t\lor v)\land z\land  (v\lor t\lor v)
\end{array}
\]
is an isomorphism of Heyting algebras by (ii), which we have already proved. We claim that $\rho(y)=v$.  Indeed,
\[\begin{array}{rcl}
\rho(y) & = & (v\lor t\lor v)\land y\land (v\lor t\lor v) \\
& = & (v\land y\land v)\lor (t\land y\land t)\lor (v\land y\land v)\\
& = & v\land y\land v=v,
\end{array}
\]
where we used strong distributivity and the fact that the elements $y$, $v$ and $t\land y\land t$ are all $\DD$-equivalent. Next we prove that $\rho (x\rightarrow y)\,\DD\, u\rightarrow y$. Denoting $\tau=v\lor t\lor v$, using the fact that $u$, $v$, $\rho(x)$ all lie in the Heyting algebra $\tau\darrow$  and that $\tau$ is the top element of $\tau\darrow$, we obtain:
 \[\begin{array}{rcl}
\rho(x\rightarrow y) & = & \rho(x)\rightarrow \rho(y) \\
& = &  \rho(x) \rightarrow v \\
& = & (v\lor(t\land  \rho(x)\land t)\lor v)   \rightarrow v  \text{ (by (NH1))}\\
& = &((v\lor t\lor v)\land (v\lor  \rho(x)\lor v)\land (v\lor t\lor v))   \rightarrow v  \text{ (by }\eqref{D2} \text{)} \\
& = &(\tau\land  (v\lor \rho(x)\lor v)\land \tau)   \rightarrow v\\
& = & (v\lor \rho(x) \lor  u\lor \rho(x) \lor v)   \rightarrow v \text{ (since } \varphi(x)\,\DD\, x\,\DD\, u\text{)} \\
& = & (v\rightarrow v) \land (\rho(x)\rightarrow v)\land (u\rightarrow v)\land (\rho(x)\rightarrow v)\land (v\rightarrow v) \text{ (in }\tau\darrow \text{)}\\
& = &  (\rho(x)\rightarrow v)\land (u\rightarrow v)\land (\rho(x)\rightarrow v)\\
& = &  \rho(x\rightarrow y) \land (u\rightarrow v)\land \rho(x\rightarrow y).\\
\end{array}
\]
That proves $\rho(x\rightarrow y)\preceq u\rightarrow v.$
Similarly we prove $u\rightarrow v\preceq \rho(x\rightarrow y)$, and $\rho (x\rightarrow y)\,\DD\, u\rightarrow y$ follows. It follows that $S/\DD$ is a Heyting algebra. It is isomorphic to $t\darrow$ as the maximal lattice image of a skew lattice is always isomorphic to any of its lattice sections.
\end{proof}

Note that given $x, y\in S$, where  $(S;\land,\lor,\rightarrow,0,t)$ is a noncommutative Heyting algebra, the element $x\rightarrow y$ equals $(y\lor (t\land x\land t)\lor y)\rightarrow y$ where the latter can be seen as the implication computed in the Heyting algebra $(y\lor t\lor y)\darrow.$

\begin{lemma}\label{lemma:normal}
Let $S$ be a normal skew lattice and let $A>B$ be comparable $\DD$-classes in $S$. Then given any $a\in A$ there exists a unique $b\in B$ s.t. $b\leq a$.
\end{lemma}

\begin{proof} 
Take any $x\in B$ and let $b=a\land x\land a$. Then $b\land a=a\land x\land a\land a=a\land x\land a=b$ and $a\land b=a\land a\land x\land a=a\land x\land a=b$. Hence $b\leq a$. Assume that $b'\in B$ also satisfies $b'\leq a.$ Using the idempotency and the normality we obtain:
\[b=b\land b'\land b=a\land b\land b'\land b\land a=a\land b'\land b\land b'\land a=a\land b'\land a=b'.
\]
\end{proof}

\begin{theorem}\label{th:imp-imp-t}
Let $S$ be a strongly distributive skew lattice with $0$ such that $S/\DD$ is a Heyting algebra. Then $S$ has a top $\DD$-class $T$. Given an element $t\in T$ define a binary operation $\rightarrow_t$  by
\[x\rightarrow_t y=y\lor u\lor y,
\]
where $u$ is the single element of the $\DD$-class $\DD_x\rightarrow \DD_y$ that lies below $t$ w.r.t. natural partial order. Then $\rightarrow_t$ satisfies the axioms (NH1)--(NH5) above. Therefore, $(S; \land, \lor,\rightarrow_t,0,t)$ is a noncommutative Heyting algebra.

On the other hand, if  $(S;\land,\lor,\rightarrow, 0, t)$  is a noncommutative Heyting algebra then $\rightarrow=\rightarrow_t$.
\end{theorem}

\begin{proof}
Denote by $\varphi:S\to S/\DD$ the homomorphism that sends $x$ to its $\DD$-class $\DD_x$. By the assumption $S/\DD$ is a Heyting algebra. Hence $T=\varphi ^{-1}(1)$ is a top $\DD$-class in $S$. Let $t\in T$ be fixed. Given $x,y\in S$ Lemma \ref{lemma:normal} yields the existence of a unique $u\in \DD_x\rightarrow \DD_y$ with the property $u\leq t$. We need to verify that the operation $\rightarrow_t$ as defined above satisfies (NH1)--(NH5).

(NH1). We have $x\rightarrow_t y=y\lor u\lor y$ but also $(y\lor (t\land x\land t)\lor y)\rightarrow_t y=y\lor u\lor y$ since $(y\lor (t\land x\land t)\lor y)\,\DD\, x\lor y$ and  $\DD_x\rightarrow \DD_y=(\DD_x\lor \DD_y)\rightarrow \DD_y$.

(NH2). $\DD_x\rightarrow \DD_x=T$ and obviously $u= t$ is the element of $T$ that is below $t$ w.r.t the natural partial order. Thus $x\rightarrow_t x=x\lor t\lor x$.

(NH3). We compute:
\begin{equation}\label{eq:imp-u}
x\land (x\rightarrow_t y)\land x=x\land (y\lor u\lor y)\land x=(x\land y\land x)\lor (x\land u\land x)\lor (x\land y\land x).
\end{equation}
In the Heyting algebra $S/\DD$ we have $\DD_x \land (\DD_x\rightarrow \DD_y)=\DD_x\land \DD_y$. That implies $x\land u\land x\,\DD\, x\land y\land x$ in $S$ and thus the expression \eqref{eq:imp-u} equals $x\land y\land x$.

(NH4). $y\land (x\rightarrow_t y)=y\land (y\lor u\lor y)$ which equals $y$ by the absorption. Likewise, $(x\rightarrow y)\land y=y$.

(NH5). Denote by $u, v$ and $w$ the elements below $t$ w.r.t. the natural partial order that lie in the $\DD$-classes $\DD_x\rightarrow (\DD_y\land \DD_z)$, $\DD_x\rightarrow \DD_y$ and $\DD_x\rightarrow \DD_z$, respectively. Computing in the lattice  $t\darrow$ yields $u=v\land w$. Thus:
\begin{eqnarray*}
x\rightarrow_t(t\land ( y\land z)\land t) = (t\land  y\land z\land t)\lor u\lor (t\land y\land z\land t)\\
= (t\land  y\land z\land t)\lor (v\land w)\lor (t\land y\land z\land t)\\
 =[ (t\land  y\land z\land t)\lor v\lor (t\land y\land z\land t)] \land [(t\land  y\land z\land t)\lor w\lor (t\land y\land z\land t)]\\
=[ (t\land  y\land t)\lor v\lor (t\land y\land  t)] \land [(t\land  z\land t)\lor w\lor (t\land z\land t)]\\
=  ( x\rightarrow_t (t\land y\land t)) \land (x \rightarrow_t (t \land  z\land t)),
\end{eqnarray*}
where the fourth equality follows because both $ [(t\land  y\land z\land t)\lor v\lor (t\land y\land z\land t)] \land [(t\land  y\land z\land t)\lor w\lor (t\land y\land z\land t)]$ and $ [(t\land  y\land t)\lor v\lor (t\land y\land  t)] \land [(t\land  z\land t)\lor w\lor (t\land z\land t)]$ are  elements of the $\DD$-class $\DD_y\land \DD_z$ that are below $t$ w.r.t. the natural partial order.

To prove the final assertion we show that given a noncommutative Heyting algebra $(S;\land,\lor,\rightarrow,0,t)$ and $x,y\in S$ the element $x\rightarrow y$ equals $y\lor u\lor y$ where $u\in \DD_{x\rightarrow y}$ s.t. $u\leq t$. We have seen that (NH4) implies $y\leq x\rightarrow y$. Moreover,  $u$ equals $t\land (x\rightarrow y)\land t$  by Lemma \ref{lemma:normal}. We obtain:
\[
\begin{array}{rcl}
y\lor u\lor y & = & y\lor (t\land (x\rightarrow y)\land t)\lor y\\
 & =& (y\lor t\lor y) \land (y\lor (x\rightarrow y)\lor y)\land (y\lor t\lor y) \\
 & = & (y\lor t\lor y)\land (x\rightarrow y)\land (y\lor t\lor y)\\
& = &
x\rightarrow y,
\end{array}
\]
where the final equality follows by Lemma \ref{lemma:d-cong}.
\end{proof}

\section{Noncommutative frames}

A symmetric skew lattice is said to be \emph{join complete} if all commuting subsets have suprema in the natural partial ordering. If $S$ is a  join complete skew lattice then $S$ has a bottom element which is obtained as the join of the empty set. By a result of Leech \cite{L2} a join complete skew lattice always has a maximal $\mathcal D$-class $T$. 

\begin{proposition}
Let $S$ be a strongly distributive, join complete skew lattice. Then:
\begin{itemize}
\item[(i)] $S$ has a bottom $0$.
 \item[(ii)] Given any $u\in S$ the set $u\darrow =\{x\in S\;|\; x\leq u\}$ is a lattice.
\item[(iii)] Given any $t$ in the top $\DD$-class $T$ the lattice $t\darrow =\{x\in S\;|\; x\leq t\}$ is a lattice section of $S$.
\end{itemize}
\end{proposition}

\begin{proof}
(i) is trivial as $0$ is obtained as the join of the empty set.

(ii) This follows from the normality of $S$ which is equivalent to $u\darrow$ being a lattice for any $u\in S$, see \cite{L4}. 

(iii) Given $x\in S$, $t\land x\land t$ is the element that lies in the intersection of $t\darrow$ and  $\DD_x$. The commutativity of $t\darrow$ follows from the normality of $\land$: $(t\land x\land t)\land (t\land y\land t)=t\land x\land y\land t=t\land y\land x\land t=(t\land y\land t)\land (t\land x\land t)$.
\end{proof}

\begin{lemma}
Let $S$ be a strongly distributive skew lattice and $\{x_i\,|\, i\in I\}$ a commuting subset. Then given any $y\in S$ the sets $\{y\land x_i\,|\, i\in I\}$  and  $\{x_i\land y\,|\, i\in I\}$ are commuting subsets.
\end{lemma}

\begin{proof}
Take $i,j\in I$. Using  normality,  idempotency of $\land$ and $x_i\land x_j=x_j\land x_i$ we obtain:
\[(y\land x_i)\land (y\land x_j)=y\land x_i\land x_j=y\land x_j\land x_i=y\land x_j\land y\land x_i=(y\land x_j)\land (y\land x_i).
\]
\end{proof}

We will use the following technical result in the proof of Theorem \ref{th-frames-cha}.

\begin{lemma}\label{lemma:tau2}
Let $S$ be a join complete strongly distributive skew lattice with a top $\DD$-class $T$, $y\in S$, $\{x_i\,|\, i\in I\}$ a commuting subset of $S$ and $t\in T$ s.t.  $x_i\leq t$ for all $i\in I$. Denoting $x=\bigvee x_i$ and $\tau=y\lor x\lor t\lor y\lor x$, the following holds:
\[(\bigvee_{i\in I} (x_i\land \tau))\land y=(\bigvee_{i\in I} x_i)\land y.
\]
\end{lemma}

\begin{proof}
Using the absorption we obtain  $\tau \land x=x$, $\tau \land x_i=x_i$ and $y\land \tau=y$. Moreover, $\tau\land x_i\land \tau\leq \tau$ for all $i$, and thus $\bigvee(\tau\land x_i\land \tau)\leq \tau$.
That yields:
\[
\begin{array}{rcl}
(\bigvee_{i\in I} (x_i\land \tau))\land y& =&(\bigvee_{i\in I} (\tau \land x_i\land \tau))\land y\land \tau    \\
  & = &\tau\land (\bigvee_{i\in I} (\tau \land x_i\land \tau))\land y\land \tau.\\
   & = &\tau\land (\bigvee_{i\in I} (x_i\land \tau))\land y\land \tau.\\
\end{array}
\]

We claim that:
\begin{equation}\label{eq:eq6}
\tau\land (\bigvee_{i\in I} (x_i\land \tau))\land y\land \tau=\tau\land (\bigvee_{i\in I} x_i)\land y\land \tau.
\end{equation}
 Both sides of \eqref{eq:eq6} lie in the same $\DD$-class and they are both below $\tau$ w.r.t. the natural partial order. Hence they must be equal by  Lemma \ref{lemma:normal}.

Finally, the right side of \eqref{eq:eq6} simplifies to
\begin{equation}\label{eq:eq6'}
 (\bigvee_{i\in I} x_i)\land y.
\end{equation}
\end{proof}

A \emph{noncommutative frame} is a strongly distributive, join complete skew lattice that satisfies the infinite distributive laws:
\begin{equation}\label{eq:inf-dist-law}
 (\bigvee_i x_i)\land y=\bigvee_i (x_i\land y) \qquad \text{and} \qquad
x\land ({\bigvee_i y_i})=\bigvee_i (x\land y_i)
\end{equation}
for all $x,y\in S$ and all commuting subsets $\{x_i\,|\, i\in I\}$  and $\{y_i\,|\, i\in I\}$.

\begin{theorem}\label{th:ncframes}
Let $S$ be a strongly distributive skew lattice with $0$ such that $S/\DD$ is a frame. Then $S$ is a noncommutative frame.
\end{theorem}

\begin{proof}
Being   a frame $S/\DD$ must be bounded and thus $S$ has a top $\DD$ class $T$. Let $\{x_i\,|\, i\in I\}$ be a commuting subset. We claim that $\bigvee x_i$ exists in $S$. Since lattice sections are maximal commuting subsets it follows that there exist $t\in T$ such that $x_i\leq t$ for all $i\in I$. Let $x$ be the single element in the $\DD$-class $\bigvee \DD_{x_i}$ that satisfies $x\leq t$. Then $x=\bigvee x_i$.

 It remains to prove that $S$ satisfies the infinite distributive laws \eqref{eq:inf-dist-law}. To see this take a commuting subset $\{x_i\,|\, i\in I\}$ and $y\in S$. Since $\{x_i\,|\, i\in I\}$ is a commuting set it is contained in a lattice section, i.e. there exists $t\in T$ s.t. $x_i\leq t$ for all $i$. Set $x=\bigvee_{i\in I} x_i$ and $\tau=y\lor x\lor t\lor y\lor x$. Since $\tau\darrow$ is a lattice section it must be isomorphic to $S/\DD$ and thus a frame. We obtain $x\leq t$,  $\tau\land x=x$, $\tau\land x_i=x_i$ and $y\land \tau=y$.
The elements $\tau \land x\land \tau$, $\tau\land y\land \tau$ and $\tau \land x_i\land \tau$ all lie in the frame $\tau\darrow$. Hence:
\begin{equation}\label{eq:inf-dist-tau}
 (\bigvee_i (\tau \land  x_i\land \tau))\land (\tau\land y\land \tau)=\bigvee_i ((\tau\land x_i\land \tau)\land (\tau\land y\land \tau))
\end{equation}
since complete Heyting algebras satisfy the infinite distributive laws. The left side of \eqref{eq:inf-dist-tau} simplifies to:
\[ (\bigvee_i (x_i\land \tau))\land (\tau\land y)
\]
which further simplifies to:
\begin{equation}\label{eq:left-inf}
(\bigvee_i (x_i\land \tau))\land y
\end{equation}
by Lemma \ref{lemma:reg}.
By Lemma \ref{lemma:tau2}  the expression in \eqref{eq:left-inf} equals
\[(\bigvee_i x_i)\land y.
\]
On the other hand, the right side of \eqref{eq:inf-dist-tau} simplifies to:
\[ \bigvee_i (x_i\land y).
\]
That proves:
\[ (\bigvee_i x_i)\land y=\bigvee_i (x_i\land y).
\]

We prove the infinite distributive law
\[x\land ({\bigvee_i y_i})=\bigvee_i (x\land y_i).
\]
in a similar way.\end{proof}

%\begin{lemma}\label{lemma:tau}
%Let $S$ be a skew lattice with a top $\DD$-class $T$, $x,y\in S$,  and $t\in T$ s.t. $x\leq t$. Set $\tau=y\lor x\lor t\lor y\lor x$. Then:
%\begin{itemize}
%\item[(i)] $\tau \land x=x$,
%\item[(ii)] $y\land \tau=y$.
%\end{itemize}
%\end{lemma}

%\begin{proof}
%Both (i) and (ii) follow by absorption.
%\end{proof}

We are now ready to prove that noncommutative frames are precisely  join complete noncommutative Heyting algebras that satisfy the infinite distributive laws.

\begin{theorem}\label{th-frames-cha}
Let $S$ be a noncommutative frame and $t$ a distinguished element in the top $\DD$-class $T$ of $S$. Given  $a,b\in S$ set:
\begin{equation}\label{eq:imp-frame}
a\rightarrow b=\bigvee_{\{x\in (b\lor t\lor b)\darrow \;|\; x\land (b\lor (t\land a\land t)\lor b)\leq b \}} x.
\end{equation}
Then $S$ is a join complete noncommutative Heyting algebra.

On the other hand, if $(S;\land ,\lor, \rightarrow, 0, t) $ is a join complete noncommutative Heyting algebra then $S$ satisfies \eqref{eq:imp-frame} and the infinite distributive laws \eqref{eq:inf-dist-law}.  
\end{theorem}

\begin{proof}  Let $S$ be a noncommutative frame. A standard result in the theory of Heyting algebras yields that frames are exactly the complete Heyting algebras. Thus the operation $\rightarrow$ as defined in the theorem yields a Heyting implication on the quotient $S/\DD$ which becomes a Heyting algebra.  In order to  prove that $S$ is a noncommutative Heyting algebra, by Theorem \ref{th:imp-imp-t} it suffices to show that $\rightarrow$ equals $\rightarrow_t$. To see this we need to show that $a\rightarrow b$ is an element in the $\DD$-class $\DD_a\rightarrow \DD_b$ that lies below $b\lor t\lor b$. (Such an element is unique by Lemma \ref{lemma:normal}.) The fact that $a\rightarrow b\in \DD_a\rightarrow \DD_b$  is clear since the operation $\rightarrow $  on $S/\DD$ is the usual Heyting implication. We simplify the notation by writing $\bigvee $ instead of $\bigvee_{\{x\in (b\lor t\lor b)\darrow \;|\; x\land (b\lor (t\land a\land t)\lor b)\leq b \}}$.  It remains to prove that $\bigvee x\leq b\lor t\lor b$. We have:
\[(\bigvee x)\land (b\lor t\lor b)=\bigvee (x\land(b\lor t\lor b))=\bigvee x,
\]
where we used $x\in (b\lor t\lor b)\darrow$. Similarly we prove $(b\lor t\lor b)\land (\bigvee x)=\bigvee x$, and $\bigvee x\leq b\lor t\lor b$ follows.

Assume now that  $(S;\land ,\lor, \rightarrow, 0, t) $ is a join complete noncommutative Heyting algebra. 
By (NH1) $a\rightarrow b$ equals $(b\lor (t\land a\land t)\lor b)\rightarrow b$, which can  be interpreted as computed in the Heyting algebra $(b\lor t\lor b)\darrow$. Proceeding with the computation in $(b\lor t\lor b)\darrow$ we obtain:
\[(b\lor (t\land a\land t)\lor b)\rightarrow b=\bigvee_{\{x\in (b\lor t\lor b)\darrow \;|\; x\land (b\lor (t\land a\land t)\lor b)\leq b \}} x.
\]
Thus $S$ satisfies  \eqref{eq:imp-frame}. 
$S$ satisfies the infinite distributive law by Theorem \ref{th:ncframes}.
%It remains to prove that $S$ satisfies the infinite distributive laws \eqref{eq:inf-dist-law}. To see this take a commuting subset $\{x_i\,|\, i\in I\}$ and $y\in S$. Since $\{x_i\,|\, i\in I\}$ is a commuting set it is contained in a lattice section, i.e. there exists $t'\in T$ s.t. $x_i\leq t'$ for all $i$. Set $x=\bigvee_{i\in I} x_i$ and $\tau=y\lor x\lor t'\lor y\lor x$. Then also $x\leq t'$ and thus $\tau\land x=x$, $\tau\land x_i=x_i$ and $y\land \tau=y$ by Lemma \ref{lemma:tau}.
%The elements $\tau \land x\land \tau$, $\tau\land y\land \tau$ and $\tau \land x_i\land \tau$ all lie in the (join) complete Heyting algebra $\tau\darrow$ and thus:
%\begin{equation}\label{eq:inf-dist-tau}
% (\bigvee_i (\tau \land  x_i\land \tau))\land (\tau\land y\land \tau)=\bigvee_i ((\tau\land x_i\land \tau)\land (\tau\land y\land \tau))
%\end{equation}
%since complete Heyting algebras satisfy the infinite distributive laws. The left side of (9) simplifies to:
%\[ (\bigvee_i (x_i\land \tau))\land (\tau\land y)
%\]
%which further simplifies to:
%\begin{equation}\label{eq:left-inf}
%(\bigvee_i (x_i\land \tau))\land y
%\end{equation}
%by Lemma \ref{lemma:reg}.
%By Lemma \ref{lemma:tau} (3) the expression in \eqref{eq:left-inf} equals
%\[(\bigvee_i x_i)\land y.
%\]
%On the other hand, the right side of (9) simplifies to:
%\[ \bigvee_i (x_i\land y).
%\]
%That proves:
%\[ (\bigvee_i x_i)\land y=\bigvee_i (x_i\land y).
%\]
%
%We prove the infinite distributive law
%\[x\land ({\bigvee_i y_i})=\bigvee_i (x\land y_i).
%\]
%in a similar way.
\end{proof}

%\section*{Acknowledgment}

\end{document}